\documentclass{mpi2015-cscpreprint}

\usepackage[american]{babel}
\usepackage{amssymb}
\usepackage{amsthm,amsmath}
\usepackage{algorithm}%
\usepackage{algorithmicx}%
\usepackage{algpseudocode}%
\usepackage{graphicx}
\usepackage{appendix}
\usepackage{multirow}

\newtheorem{theorem}{Theorem}[section]

\newtheorem{lemma}[theorem]{Lemma}

\numberwithin{equation}{section}

\begin{document}

\title{On classical solutions in the stabilization problem for nonholonomic control systems with time-varying feedback laws}

\author[1,3]{Alexander~Zuyev}
\author[2,3]{Victoria Grushkovskaya}
\affil[1]{Max Planck Institute for Dynamics of Complex Technical Systems, 39106 Magdeburg, Germany}
\affil[2]{Department of Mathematics, University of Klagenfurt, 9020 Klagenfurt am W\"orthersee, Austria}
\affil[3]{Institute of Applied Mathematics and Mechanics, National Academy of Sciences of Ukraine}



\abstract{
We consider the stabilization problem for driftless control-affine systems under the bracket-generating condition.
In our previous works, a class of time-varying feedback laws has been constructed to stabilize the equilibrium of a nonholonomic system under rather general controllability assumptions.
This stabilization scheme is based on the sampling concept, which is not equivalent to the classical definition of solutions for the corresponding nonautonomous closed-loop system.
In the present paper, we refine the previous results by
presenting sufficient conditions for the convergence of classical solutions of the closed-loop system to the equilibrium. Our theoretical findings are applied to a multidimensional driftless control-affine system and illustrated through numerical simulations.}

\maketitle

\section{INTRODUCTION}\label{introsec}
Consider a driftless control-affine system
\begin{equation}\label{sys_affine}
\dot x = \sum_{j=1}^m u_j f_j(x),\quad x\in D\subset {\mathbb R}^n,\;  u\in {\mathbb R}^m,\;m<n,
\end{equation}
with the state $x=(x_1,x_2, ..., x_n)^\top$ and the control $u=(u_1,u_2,...,u_m)^\top$.
We assume that the vector fields  $f_1$, ..., $f_m$ are smooth in the domain $D$, with $0\in D$, and that $${\textrm rank}\left( f_1(0), f_2(0),...,f_m(0)\right) = m.$$
This class of systems includes, in particular, mathematical models of underactuated robots with nonholonomic constraints.
It is well-known that the equilibrium $x=0$ of the underactuated system~\eqref{sys_affine} is not stabilizable by a regular time-invariant feedback law $u=h(x)$ due to the violation of Brockett's condition~\cite{Bro83}.

The stabilization problem for underactuated systems under nonlinear controllability conditions has been extensively studied. Without pretending to be exhaustive, we mention stabilizability results for general controllable systems using
time-invariant discontinuous controllers~\cite{Clar97}
and
time-varying feedback laws~\cite{Cor92},~\cite[Chap.~11]{coron2007control}, as well as averaging techniques with fast oscillating inputs~\cite{Bom13} and control design with transverse functions~\cite{morin2004practical,morin2009control}.

In~\cite{zhang2022new}, a class of nonholonomic systems with external disturbances is considered, and constrained stabilizing feedback laws are proposed using state transformations and dynamic surface control. These controllers are experimentally tested on a wheeled mobile robot.
A chained nonholonomic system in cascaded form with two subsystems is considered in~\cite{gao2024fas}, accounting for model uncertainties and disturbances. The stabilization problem is addressed using a state feedback law, where output control is applied to the fully actuated subsystem, and an exponentially stable component is designed.
In~\cite{liu2024event}, a fuzzy logic approach is applied to approximate uncertainties in a kinematic nonholonomic system to address the stabilization problem with output constraints. The considered constraints are satisfied by combining the design of a barrier Lyapunov function with an event-triggered scheme, which is applied to a wheeled mobile robot.

In the paper~\cite{ZuSIAM}, a family of time-varying feedback laws was proposed to solve the exponential stabilization problem {for} nonholonomic systems {with  degree of nonholonomy~2.} The coefficients of these controllers are defined by solving an auxiliary system of nonlinear algebraic equations, and a simplified control design procedure was developed in~\cite{GZ23} using the inverse of the matrix that appears in the controllability rank condition.
Additionally, an algorithm for constructing stabilizing controllers for high-order nonholonomic systems was recently reported in~\cite{grushkovskaya2024design}.
In the above-mentioned papers~\cite{GZ23,ZuSIAM,grushkovskaya2024design},
{the behavior of the closed-loop system is analyzed in the sense of sampled-data solutions,
leaving open the question of the asymptotic behavior of classical (or Carath\'eodory) solutions under the proposed Lie-bracket-based oscillatory feedback laws. It is important to note that the available characterization of the sampled system (based on $\varepsilon$-discrete updates) does not guarantee convergence in the classical solution sense.
The present paper addresses this gap by rigorously proving that time-varying controllers, whose oscillatory components are inspired by those in~\cite{GZ23,ZuSIAM}, ensure convergence of classical solutions to the equilibrium. This establishes novel stability results and enhances the theoretical foundation of Lie-bracket-based control design.}

The subsequent exposition is organized as follows.
Section~\ref{gensec} presents our control design scheme, which is based on a weak control Lyapunov function and establish strict decrease conditions using suitable oscillating controllers. The resulting convergence results are formulated in Theorem~1, while Theorem~2 {highlights the novelty of our design by incorporating Lie-bracket structure of the considered system.}
In Section~\ref{secex}, the proposed control methodology is applied to a multidimensional nonholonomic integrator, {illustrating the effectiveness of the design.}
Finally, concluding remarks are provided in Section~\ref{secconc}.

\section{CONTROL DESIGN}\label{gensec}
Let us consider a positive definite function $V\in C^2(D)$  and denote by $\dot V(x,u)=\sum\limits_{j=1}^m u_j L_{f_j} V(x)$  the time derivative of $V(x)$ along the trajectories of~\eqref{sys_affine},
where $L_{f_j}V(x):=\sum\limits_{i=1}^n \tfrac{\partial V(x)}{\partial x_i} f_{ji}(x)=\nabla V(x) f_j(x)$ stands for the directional derivative of $V(x)$ along $f_j(x)$ (the gradient $\nabla V(x)=\Big(\tfrac{\partial V}{\partial x_1},...,\tfrac{\partial V}{\partial x_n}\Big)$ is treated as row vector in these notations). Note that any such $V(x)$ is a {\em weak} control Lyapunov function (weak CLF) because of the driftless form of system~\eqref{sys_affine}, i.e. for each $x\in D$, there exists a $u_x\in {\mathbb R}^m$ such that $\dot V(x,u_x)\le 0$. Note that there are no {\em strict} control Lyapunov functions $V(x)$ for system~\eqref{sys_affine}  under our assumptions because of Brockett's~\cite{Bro83} and Artstein's~\cite{Artstein} theorems (see also~\cite{Son89}). It means that $\dot V(x,u)$ cannot be made negative definite by any choice of controls.
Thus, for any positive definite $V$ and any neighborhood of the origin in $\mathbb R^n$, there are points $x\neq 0$ with $\dot V(x,u)\ge 0$ for all $u\in{\mathbb R}^m$.

We  propose a solution of the asymptotic stabilization problem by constructing a family of time-varying feedback controls $u=h^\varepsilon(x,t)$ depending on a small parameter $\varepsilon>0$. Even though the time derivative of $V$ along the classical solutions $x(t)$ of the corresponding closed-loop system is not negative definite, we will derive sufficient conditions for the finite difference $V(x(t_0+\varepsilon))-V(x(t_0))$ to be negative along nontrivial solutions, provided that $\varepsilon$ is small enough. Then, under additional technical assumptions, the above property implies the convergence of $x(t)$ to $0$ as $t\to +\infty$.

For the sake of brevity of presentation, we restrict ourselves to a particular case of
{systems~\eqref{sys_affine} with degree of nonholonomy 2 (see, e.g.,~\cite{bellaiche2005geometry} for the definition of the degree of nonholonomy). In this case, the controllability of system~\eqref{sys_affine} is established using the Lie brackets of the vector fields $f_j$, without the need to consider higher-order (iterated) Lie brackets.}
 Namely, let
$S\subset \{1,2,...,m\}^2$ be a set
{of index pairs}
$I=(ij)$ such that {its cardinality is} $\vert S \vert =n-m$, and
\begin{equation}\label{bracket_gen}
{\textrm span}\{f_k(x), f^I(x) \,\vert\, k {\in\{1,\dots,m\}},\, I\in S \} = {\mathbb R}^n\;{\forall  x\in D},
\end{equation}
where $f^I(x):= [f_i,f_j](x)$ is the Lie bracket of the vector fields $f_i(x)$ and $f_j(x)$
{corresponding to the index pair
$I=(ij)$, defined by}
$
[f_i,f_j](x):= \dfrac{\partial f_j(x)}{\partial x}f_i(x) - \dfrac{\partial f_i(x)}{\partial x}f_j(x)
$.\\
We then construct a time-varying feedback control $u=h^\varepsilon(x,t)$ {with  components}
\begin{equation}\label{feedback_general}
h_k^\varepsilon(x,t) = v_k^0(x) + { \gamma} \sum_{I\in S} v_k^I(x)\phi_k^{I,\varepsilon}(t),\; k=1,...,m,
\end{equation}
where the time-independent part $v^0(x)$ is designed {to ensure $\dot V(x,v^0(x))\le 0$ for} a weak CLF candidate $V$,
$\gamma>0$ is a gain parameter,
and the time-varying part $v^I(x)\phi^{I,\varepsilon}(t)$ approximately drives the system along the Lie bracket $\pm f^I(x)$  to ensure a decrease in $V(x(t))$  over discrete
 $\varepsilon$-time shifts, even when $\dot V$ is not negative definite. The functions $\phi_k^{I,\varepsilon}(t)$ are defined in the following way: if $I=(ij)\in S$, then
\begin{equation}\label{phi_general}
\hspace{-0.5em}\begin{aligned}
&\phi_i^{I,\varepsilon} (t) {=} 2\sqrt{\tfrac{\varkappa_I \pi }{\varepsilon}}\cos (\varkappa_I \omega t), \phi_j^{I,\varepsilon} (t) {=} 2\sqrt{\tfrac{\varkappa_I \pi }{\varepsilon}}\sin (\varkappa_I \omega t),\\
&\phi_k^{I,\varepsilon} (t) = 0\; \text{for all}\; k\notin \{i,j\},
\end{aligned}
\end{equation}
where $\omega=\frac{2\pi}{\varepsilon}$ and $\varkappa_I$ is a positive integer frequency multiplier. The latter integer multipliers should be chosen to avoid resonances between the control components that generate different Lie brackets, i.e.
\begin{equation}\label{nonres}
\varkappa_I \neq \varkappa_{I'}\quad \text{for all}\;\; S\ni I\neq I' \in S.
\end{equation}
In particular, we can take $\varkappa_{I_1}=1$, $\varkappa_{I_2}=2$, ...,  $\varkappa_{I_{n-m}}=n-m$ for $S=\{I_1,I_2,...,I_{n-m}\}$.
The closed-loop system~\eqref{sys_affine} with controls $u_k=h^\varepsilon_k(x,t)$ of the form~\eqref{feedback_general} reads as
\begin{equation}\label{closed-loop}
    \dot x = g_0(x) + \gamma \sum_{I\in S}\sum_{k=1}^m g_k^I(x)\phi_k^{I,\varepsilon} (t)
\end{equation}
\vspace{-1em}
\begin{equation}\label{g_fields}
\text{with }\quad g_0(x) = \sum_{k=1}^m v_k^0(x) f_k(x),\; g_k^I(x) = v_k^I(x) f_k(x),
\end{equation}
where, similarly to~\eqref{phi_general}, we assume that, for each $I=(ij)\in S$:
$
v_k^I(x) = 0 \; \text{and}\;  g_k^I(x) =0 \;\text{for all}\; k\notin\{i,j\}
$.

{To analyze the behavior of solutions of system~\eqref{closed-loop}, we will exploit the  Chen--Fliess expansion.
This leads to the result below, which is directly based on~\cite[S.~40.2]{La95},~\cite[Lemma~5.1]{ZG17}.
\begin{lemma}
 \textit{Let $D\subseteq\mathbb R^n$ be a domain,     $\psi_i{\in} C^2(D;\mathbb R^n)$, $i{\in}\{0,1,\dots ,m\}$, and let $x(t){ \in} D$, $t {\in} [0,\tau]$, $\tau{>}0$, be a solution of the system $\dot{x} = \psi_0(x)+\sum_{i=1}^{\ell} \psi_i(x)w_i(t)$ with $w {\in} C([0,\tau]; \mathbb{R}^m)$ and $x(0){ = }x^0 {\in} D$. Then, for any $t {\in} [0, \tau]$, $x(t)$ can be represented by the Chen--Fliess series as}
{\small $$
\begin{aligned}
&x(t) = x^0 + \sum_{i=0}^\ell \psi_i(x^0)\int_0^t w_i(s_1)ds_1\\
 &+ \sum_{i_1,i_2=0}^{\ell}  L_{\psi_{i_2}} \psi_{i_1}(x^0)
\int_0^t \int_0^{s_1}  w_{i_1}(s_1)   w_{i_2}(s_2) \, ds_2 ds_1 + \rho(t),   \\
&\text{with }\rho(t) = \sum_{i_1,i_2,i_{3}=0}^{\ell}\int_0^t \int_0^{s_1}\int_0^{s_2}   L_{\psi_{i_{3}}} L_{\psi_{i_2}}\psi_{i_1}(x(s_{3}))
\\
&\qquad\times w_{i_1}(s_1) w_{i_2}(s_2) w_{i_3}(s_3) \, ds_{3} ds_2 ds_1.
\end{aligned}$$}
\emph{For notational convenience, we define
$w_0\equiv 1$.}
\end{lemma}
}

If the vector fields $g_0$, $g_k^I$ are of class $C^2$ and the functions $\phi_k^{I,\varepsilon} (t)$ are defined by~\eqref{phi_general}, then  the  Chen--Fliess expansion {from Lemma~1} for the solution $x(t)\in D$ of~\eqref{closed-loop}, $t\in [0,\varepsilon]$, with the initial data $x(0)=x^0\in D$   can be expressed {at $t=\varepsilon$ as follows~(see formula~(77) in~\cite{ZG17})}:
\begin{equation}\label{Chen-Fliess}
\begin{aligned}
&x(\varepsilon) = x^0 + \varepsilon  g_0(x^0) + \frac{\varepsilon^2}{2} L_{g_0}g_0(x^0)- \varepsilon^{3/2} {\gamma}\sqrt{\pi} \\
&\ \times \sum_{{I=(ij)\in S}} \frac{[{g_0},g_j^I](x^0)}{\sqrt{\varkappa_{I}}} + \varepsilon {\gamma}^2 \sum_{{I=(ij)\in S}} [g_i^I,g_j^I](x^0)\\
&\qquad + O(\varepsilon^{3/2})= x^0 + \varepsilon \bigl( g_0(x^0) \\
&+{\gamma}^2 \sum_{{I=(ij)\in S}} [g_i^I,g_j^I](x^0) \bigr) + \varepsilon^{3/2} R(x^0,\varepsilon,\gamma).
\end{aligned}
\end{equation}
{Assumption~\eqref{nonres} is crucial in the derivation of~\eqref{Chen-Fliess} to avoid resonances in the Lie bracket approximations and ensure orthogonality between oscillatory components with different index pairs in $S$.}
For each compact $ D'\subset D$ and each $\gamma>0$, there is a $\tilde\varepsilon>0$ such that $\|R(x^0,\varepsilon,\gamma)\|\le C_R$ for all $(x^0,\varepsilon)\in \tilde D\times (0,\tilde\varepsilon]$, where the constant $C_R$ can be obtained from Lemma~5.1 of~\cite{ZG17}.
The Lie brackets appearing in~\eqref{Chen-Fliess} can also be represented in terms of the original vector fields $f_i$, $f_j$ with $(ij)=I$:
\begin{equation}\label{brackets_fg}
[g_i^I,g_j^I] = v_i^I v_j^I [f_i,f_j] + v_i^I f_j \nabla v_j^I f_i -  v_j^I f_i \nabla v_i^I f_j,
\end{equation}
under the convention that the gradient $\nabla v_j^I(x)$ is a row vector.
Formula~\eqref{Chen-Fliess} allows to express the increment of $V(x(t))$ along the trajectories of~\eqref{closed-loop}:
\begin{equation}
    \label{V_increment}
    \frac{V(x(\varepsilon)) - V(x^0)}{\varepsilon} = W(x^0) + \varepsilon^{1/2} r(x^0,\varepsilon,\gamma),
\end{equation}
where
\begin{equation}
    \label{W_gen}
    W(x) = L_{g_0}V(x) + \gamma^2 \sum_{I=(ij)\in S} L_{[g_i^I,g_j^I]}V(x),
\end{equation}
and $r(x^0,\varepsilon,\gamma)$ is a bounded remainder which can be estimated in terms of $R(x,\varepsilon,\gamma)$ and the  derivatives of $V$. A key idea of our control design is to define the functions $v^I_k(x)$ and gain $\gamma$ to ensure that $W(x)$ is negative definite.
The main result in this area is summarized below.
\begin{theorem}\label{prop_general}
{\em Let the function $V:D\to \mathbb R$ and vector fields $g_0, g^I_k: D\to{\mathbb R}^n$, $k {\in\{1,\dots,m\}}$, $I\in S$, be twice continuously differentiable in $D$,  $\gamma$ be a positive constant, $D'\subseteq D$ be a closed bounded domain such that $0\in {\textrm int} \, D'$,  and let the following conditions be satisfied:
\begin{itemize}
\item[i)] $V(x)$ is positive definite in $D'$;
\item[ii)] $W(x)$ given by~\eqref{W_gen} is {negative definite} in $D'$;
\item[iii)] there exist an $\bar\varepsilon>0$ and a $c_r>0$ such that
$$
\begin{aligned}
|r(x,\varepsilon,\gamma)|\le -c_r W(x)\;\;\text{for all}\;x\in D',\;\varepsilon\in (0,\bar\varepsilon];
\end{aligned}
$$
\item[iv)] $g_0(0)=0$ and $g^I_k(0)=0$ for all $k {\in\{1,\dots,m\}}$, $I\in S$;
\end{itemize}
Then there exists a $\delta>0$ such that each solution $x(t)$ of system~\eqref{closed-loop}  with the initial condition $x(0)\in B_\delta(0)\subset D'$ tends to zero as $t\to+\infty$, provided that the time-varying functions $\phi_k^{I,\varepsilon}$ are given by~\eqref{phi_general} under assumption~\eqref{nonres}, and the parameter   $\varepsilon>0$ is {sufficiently small}.}
\end{theorem}

\begin{proof}
Given a function $V$ satisfying condition~i), denote its level set as
$
\mathcal L_c=\{x\in D:V(x)\le c\}$ for $c\ge 0$.
Let $d>0$ and $\delta>0$ be such that
$
\overline{B_\delta(0)}\subset \mathcal L_d\subset {\textrm int}\, D'
$,
where $\overline{B_\delta (0)}$ denotes the closed ball of radius $\delta$ centered at~$0$.

Let $\varepsilon_0\in (0,\bar\varepsilon]$ be small enough to guarantee that the solutions of system~\eqref{closed-loop} with $x(0)=x^0\in \overline{B_\delta(0)}$ are well-defined in  $D'$ on the time interval $t\in [0,\varepsilon_0]$. Throughout the proof, we assume $\varepsilon\in(0,\varepsilon_0]$.

The first step of the proof involves ensuring that the solutions of system~\eqref{closed-loop} with $x^0\in  \mathcal L_d$ and a sufficiently small $\varepsilon$ evolve in $D'$ for $t\in[0,\varepsilon]$, and
obtaining some a priori estimates.
Denote
$L_0{=}\sup\limits_{x\in D'}\left\|\frac{\partial g_0(x)}{\partial x}\right\|,\,L_g{=}\max\limits_{I\in S,k\in\{1,\dots,m\}}\sup\limits_{x\in D'}\left\|\frac{\partial g_k^I(x)}{\partial x}\right\|,$
then,  taking into account  $ g_0(0){=}g_k^I(0){=}0$ in~iv), we obtain
\begin{equation}
\label{g_estimate}
  \|g_0(x)\|\le L_0\|x\|,\,\|g_k^I(x)\|\le L_g\|x\|,
\end{equation}
for all $x\in D'$, $I\in S$, $k\in\{1,\dots,m\}$. Furthermore, note that
$
\max\limits_{t\in[0,\varepsilon]}\sum\limits_{I\in S}\sum\limits_{k=1}^m|\phi_k^{I,\varepsilon}(t)|\le{c_\phi}\varepsilon^{-1/2}
$
with $c_\phi=2\sqrt{2\pi}\sum_{I\in S}\sqrt{\varkappa_I}$.
From the integral representation of solutions of system~\eqref{closed-loop},
$$
x(t)=x^0+\int_0^t\big(g_0(x(\tau)) + \gamma \sum_{I\in S}\sum_{k=1}^m g_k^I(x(\tau))\phi_k^{I,\varepsilon} (\tau)\big)d\tau.
$$
Thus,  for any $t\in[0,\varepsilon]$,
$
\|x(t)\|\le \|x^0\|+\Big(L_0+{L_g c_\varphi}{\varepsilon}^{-1/2}\Big)\int_0^t\|x(\tau)\|d\tau.
$
Then the  Gr\"onwall--Bellman inequality implies
\begin{equation}
    \label{x_estimate}
    \|x(t)\|\le c_x\|x^0\| \text{ for all }t\in[0,\varepsilon],
\end{equation}
where $c_x=e^{L_0\varepsilon+L_g c_\phi\sqrt\varepsilon}$. Taking $\varepsilon_1>0$ as the minimum of $\varepsilon_0$ and the positive solution of the equation $L_0\varepsilon+L_g c_\phi\sqrt\varepsilon=\ln\frac{\sup_{x\in D'}\|x\|}{\sup_{x\in\mathcal L_d}\|x\|}$, we conclude that the solutions of system~\eqref{closed-loop} with $x^0\in  \mathcal L_d$ and $\varepsilon\in(0,\varepsilon_1]$ stay in $D'$ for all $t\in [0,\varepsilon]$.

The next step is to show that there exists an $\varepsilon_2\in (0,\varepsilon_1]$ such that  the function $V$ satisfies the property $V(x(\varepsilon))<V(x^0)$
along the solutions of system~\eqref{closed-loop} with $\varepsilon\in (0,\varepsilon_2)$ . For this purpose, we exploit  representation~\eqref{V_increment} and assumptions ii)--iii). Namely, let the constants $\bar\varepsilon$ and $c_r>0$ be defined from the assumptions ii) and iii). Then  representation~\eqref{V_increment} implies
\begin{equation}\label{V_decr}
V(x(\varepsilon))\le V(x^0)+\varepsilon W(x^0)(1-\sqrt\varepsilon c_r).
\end{equation}
Putting $\varepsilon_2={\textrm min}\{ c_r^{{-2}},\varepsilon_1\}$, we ensure
$V(x(\varepsilon))<V(x^0)$ for
any $x^0\in\mathcal L_d\setminus\{0\}$ and
any $\varepsilon\in(0,\varepsilon_2)$. In particular, $x(\varepsilon)\in \mathcal L_d$.

Repeating the above steps for $x(\varepsilon)$, $x(2\varepsilon)$, \dots, we conclude that $x(t)\in D'$ for all $t\in[0,\infty)$ and obtain the monotonically decreasing sequence $\{V(x(j\varepsilon))\}_{j\in\mathbb N}$ with $\inf_j V(x(j\varepsilon))={V_0\ge 0}$.
{
To prove that $V_0=0$, assume the contrary: let $V_0>0$, then the level set ${\cal M}_0=\{x\in D\,\vert \, V(x)=V_0\} \subset {\cal L}_{V(x^0)}$ is compact and separated from $0$, since  $V$ is positive definite.
By applying inequality~\eqref{V_decr}, we conclude that $\lim_{j\to \infty}(V (x(j\varepsilon))-V( x( (j-1)\varepsilon ) ) )\le \varepsilon (1-\sqrt{\varepsilon} c_r) \sup_{\xi\in {\cal M}_0} W(\xi)<0$ because $W$ is negative definite.
The latter inequality implies that $V (x(j\varepsilon))\to -\infty$ as $j\to \infty$,
which contradicts the property $V(x)\ge 0$.
This contradiction  proves that $V_0=0$.
}
This implies $\lim_{j\to\infty}V(x(j\varepsilon))=0$, and therefore, $\lim_{j\to\infty}\|x(j\varepsilon)\|=0$ because of the continuity and positive definiteness of $V$. Together with~\eqref{x_estimate}, this yields $\lim_{t\to\infty}\|x(t)\|=0$.


\end{proof}

{
{\em Remark.}
As it can be seen in the proof of Theorem~1, the convergence of classical solutions to zero holds for any $x(0)\in B_\delta(0)$ and any $\varepsilon\in (0,\varepsilon_2)$, where $\varepsilon_2>0$ is defined in the proof above.}

The next theorem justifies the applicability of the controls, proposed previously in papers~\cite{ZuSIAM,ZG17,GZ23,grushkovskaya2024design} in the context of sampling, for stabilizing the classical solutions.
\begin{theorem}
\textit{{Let assumption~\eqref{bracket_gen} be satisfied, and let $F(x)=\left(f_1(x),\dots , f_m(x),  f_I(x)_{I\in S}\right)$ denote the $n\times n$-matrix whose columns are composed of  the vector fields from~\eqref{bracket_gen}.}
Given a positive definite function $V\in C^2(D)$ such that $\nabla V(x)=0$ only at $x=0$,
{define the vector function $\tilde v:D\to {\mathbb R}^n$
by the rule
{\small
\begin{equation}\label{v_form0}
\tilde v(x)=(v_1^0(x), \dots v_m^0(x), {\tilde v}^I(x)_{I\in S})^\top=-F^{-1}(x)(\nabla V(x))^\top,
\end{equation}}
and define the functions  $v^I_k$ for each $I\in S$ in~\eqref{feedback_general} as follows:
%
\begin{equation}\label{v_form}
\begin{aligned}
& v_i^I(x)=\sqrt{|{\tilde v}^I(x)|},\\
& v_j^I(x)=\sqrt{|{\tilde v}^I(x)|}{\textrm sign}({\tilde v}^I(x))\;\;\text{for}\;\; (ij)=I,\\
&v_k^{I} (x) = 0\;\; \text{for all}\;\; k\notin \{i,j\},\; (ij)=I.
\end{aligned}
\end{equation}}
Assume that there exists a $\gamma>0$ and a closed bounded domain $D'\subseteq D$ with $0\in {\textrm int}\, D'$ such that the following conditions hold:
\begin{itemize}
    \item[c1)] $\sup\limits_{x\in D'\setminus\{0\}}\dfrac{\nabla V(x)\Phi(x,\gamma)}{\|\nabla V(x)\|^2}<1$, where
$$
\begin{aligned}
&\hspace{-2em}\Phi(x,\gamma)=\sum\limits_{I=(ij)\in S}\Big((\gamma^2-1)[f_i^I,f_j^I](x)\tilde v^I(x)+\gamma^2f_j(x)\\
&\times\nabla|\tilde v^I(x)|f_i(x){\textrm sign}(\tilde v^I(x))-f_i(x)\nabla|\tilde v^I(x)|f_j(x)\Big);
\end{aligned}
$$
\item [c2)] $\sup\limits_{x\in D'}\dfrac{|r(x,\varepsilon,\gamma)|}{\|\nabla V(x)\|^2}<\infty$ for all $\varepsilon\in (0,\bar\varepsilon]$ for some $\bar \varepsilon>0$, where $r(x,\varepsilon,\gamma)$ is the remainder in representation~\eqref{V_increment}.
\end{itemize}
Then, there exists a $\delta>0$ such that each solution $x(t)$ of {the closed-loop system~\eqref{sys_affine},~\eqref{feedback_general}} with the initial condition $x(0)\in B_\delta(0)\subset D'$ tends to zero as $t\to+\infty$, provided that {the components of~\eqref{feedback_general} are given by~\eqref{v_form0},~\eqref{v_form} and}~\eqref{phi_general} under assumption~\eqref{nonres}, and that the parameter $\varepsilon>0$ is sufficiently small.}
\end{theorem}
\begin{proof}
The proof is similar to the proof of Theorem~1.
Under the assumptions of Theorem~2, the function $W$ defined by~\eqref{W_gen} can be represented  as
$$
\begin{aligned}
    W(x)&=\nabla V(x)\Big(g_0(x)+ \gamma^2\sum_{I=(ij)\in S}\big([f_i^I,f_j^I](x)\tilde v^I(x)\\
  +f_j(x)&\nabla|\tilde v^I(x)|f_i(x){\textrm sign}(\tilde v^I(x))-f_i(x)\nabla|\tilde v^I(x)|f_j(x)\Big)\\
    &= \nabla V(x)F(x)\tilde v(x)+\nabla V(x){ \Phi(x,\gamma)}\\
    &=-\|\nabla V(x)\|^2+\nabla V(x) \Phi(x,\gamma).
\end{aligned}
$$
By c1), there exist $\gamma>0$ and $\sigma\in [0,1]$, such that
 $$
 W(x)\le -(1-\sigma)\|\nabla V(x)\|^2\text{ in }D'.
 $$
Additionally,  c2) implies  $|r(x,\varepsilon,\gamma)|\le \tilde c_r \|\nabla V(x)\|^2 $ in $D'$. Thus, representation~\eqref{V_increment} yields
$$
V(x(\varepsilon))\le V(x^0)-\varepsilon\|\nabla V(x^0)\|^2 \big(1-\sigma-\sqrt\varepsilon\tilde c_r\big).
$$
If $\varepsilon\in\big(0,{\textrm min}\{((1-\sigma)/\tilde c_r)^2,\bar \varepsilon\}\big)$, then $V(x(\varepsilon))<V(x^0)$ for $x(\varepsilon)\ne x^0$, and the rest of the proof goes along the same line as the proof of Theorem~1.
\end{proof}

Under more general assumptions, the condition~ii) of Theorem~1 can be established by the following technical lemma.

\begin{lemma}
\emph{Let a function $V: D\to\mathbb R$ be differentiable such that $\nabla V(0)=0$. Assume that:
\begin{itemize}
  \item[a)] $
    \alpha(x):=L_{g_0}V(x)  \le 0$ for all $x\in D$;
\item[b)] $\beta(x) := \sum\limits_{I=(ij)\in S} L_{[g_i^I,g_j^I]}V(x)<0$ for all $x\in M_0\setminus\{0\}$,
where
$
M_0 = \{x\in D\, \vert\, \alpha(x)=0\};
$
\item[c)]
$\sup\limits_{x \in D \setminus M_0} \left\{-\dfrac{\beta(x)}{\alpha(x)}\right\} < \infty.$
\end{itemize}
  Then there exists a $\gamma>0$ such that the function $W(x)$ defined by~\eqref{W_gen} is negative definite. }
\end{lemma}
\begin{proof}
    Indeed, under assumption~b), the function
    $
    W(x)=\alpha(x)+\gamma^2\beta(x)
    $
is strictly negative in $M_0\setminus\{0\}$ for each $\gamma\neq 0$, and $W(0)=0$ because $\nabla V(0)=0$.
It remains to show that $W(x)<0$ for all $x\in D\setminus M_0$.
Let $c_{\alpha\beta}=\sup\limits_{x \in D\setminus M_0} \left\{-\frac{\beta(x)}{\alpha(x)}\right\} <\infty$.
We consider the two possible cases.

Case~1: $c_{\alpha\beta}\le 0$. Then $W(x)<0$ for all real $\gamma$ and all $x\in D\setminus M_0$.

Case~2. $c_{\alpha\beta}> 0$. In this case, $W(x)<0$ for all $x\in D\setminus M_0$, provided that $\gamma^2<1/c_{\alpha\beta}$.

\end{proof}

%
%

We will illustrate the proposed control design scheme with a multidimensional nonholonomic system in the following section.

\section{CASE STUDY}\label{secex}



Consider the following  underactuated nonlinear control system:
\begin{equation}\label{brockett}
\dot x = \sum_{j=1}^4 u_j f_j(x),\quad x\in{\mathbb R}^{10},\; u \in {\mathbb R}^4,
\end{equation}
where $u =(u_1,u_2,u_3,u_4)^\top$, $x=(x_1,x_2,\dots,x_{10})^\top$,
$$
\begin{aligned}
f_1(x) &= (1, 0, 0, 0,-x_2, -x_3,-x_4, 0,0,0 )^\top,\\
f_2(x) &= ( 0, 1, 0, 0, x_1, 0, 0, -x_3, -x_4, 0)^\top,\\
f_3(x) &= (0, 0, 1, 0, 0, x_1, 0, x_2, 0, -x_4 )^\top,\\
f_4(x) &= (0, 0, 0, 1, 0, 0, x_1, 0, x_2, x_3 )^\top.
\end{aligned}
$$
System~\eqref{brockett} belongs to the class of nilpotent systems described by~R.~W.~Brockett in~\cite{Bro82}. {This nilpotent property means that all iterated Lie brackets of the vector fields $f_j$ (i.e., the Lie brackets of length greater than 2)  vanish. Thus, the Chen--Fliess expansion of the solutions of~\eqref{brockett} contains only terms with $f_j$ and their mutual Lie brackets, without any higher-order terms.}
It is easy to show that system~\eqref{brockett} satisfies the controllability rank condition with the Lie brackets $f^{(ij)}(x):=[f_i,f_j](x)$ for $
(ij)\in S$,
$
S=\{(1 2),\,(1 3),\,(1 4),\,(2 3),\,(2 4),\,(3 4)\}.
$
Indeed,
$f^{(12)}(x) = 2e_5$, $f^{(13)}(x) =  2e_6$,
$f^{(14)}(x) =  2e_7$, $f^{(23)}(x) =  2e_8$,
$f^{(24)}(x) =  2e_9$, $f^{(34)}(x) =  2e_{10}$,
where $e_j\in\mathbb R^{10}$ denotes the unit column vector with $1$  at the $j$-th entry.
Thus,
${\textrm span} (f_1(x),\dots,f_4(x), f^I(x)_{I\in S})={\mathbb R}^{10}
\;\text{for each}\; x\in \mathbb R^{10}.
$
System~\eqref{brockett} is not stabilizable by a continuous time-invariant feedback law of the form $u=h(x)$ because it does not satisfy Brockett's necessary condition~\cite{Bro83}.
It should be mentioned that Brockett's condition remains necessary for a class of partial stabilization problems~\cite{Zu99}, which poses challenges for stabilizing control design with respect to certain variables.
Note that system~\eqref{brockett} is nilpotent, so that higher-order terms with iterated Lie brackets in the Chen--Fliess expansion of~\eqref{brockett} vanish.

We take the Lyapunov function candidate as
$
V(x) = \frac12\sum\limits_{j=1}^4x_j^2+\frac{1}{2{p}}\sum\limits_{j=5}^{10}x_j^{2{p}},\,{p}\ge 1,
$
and define the following time-varying feedback control in the form~\eqref{feedback_general}:
\begin{equation}\label{brockett_feedback}
\begin{aligned}
&u_1 {= }v_1^{0}(x) + \gamma \sum_{k=2}^4v_1^{(1k)} (x) \phi_1^{(1k),\varepsilon} (t),\\
&u_2 {=} v_2^{0}(x) + \gamma \sum_{k=3}^4v_2^{(2k)} (x) \phi_2^{(2k),\varepsilon} (t) + \gamma v_2^{(12)} (x) \phi_2^{(12),\varepsilon} (t),\\
&u_3 {=} v_3^{0}(x)  + \gamma \sum_{k=1}^2v_3^{(k3)} (x) \phi_3^{(k3),\varepsilon} (t)+ \gamma v_3^{(34)} (x) \phi_3^{(34),\varepsilon} (t),\\
&u_4 {=} v_4^{0}(x) + \gamma \sum_{k=1}^3v_4^{(k4)} (x) \phi_4^{(k4),\varepsilon} (t),
\end{aligned}
\end{equation}
where {$v_k^{0}(x) = - x_k$, $k\in\{1,2,3,4\}$, and $\phi_k^{I,\varepsilon}(t)$ are defined by formulas~\eqref{phi_general} with $\varkappa_{(12)}=1$, $\varkappa_{(13)}=2$, $\varkappa_{(14)=3}$, $\varkappa_{(23)}=4$, $\varkappa_{(24)}=5$, $\varkappa_{(34)}=6$.}
The resulting closed-loop system~\eqref{brockett},~\eqref{brockett_feedback} takes the form
\begin{equation}\label{brockett_closedloop}
\dot x = g_0(x) + \gamma\sum_{I\in S}\sum_{k=1}^4g_k^I(x)\phi_k^{I,\varepsilon} (t),
\end{equation}
with
$
g_0(x) = -( x_1, x_2, x_3, x_4, 0, 0, 0, 0, 0, 0)^\top,
$
and  $g_k^I(x) = v_k^I (x) f_k(x),\; k\in\{1,2,3,4\}$.
\begin{figure*}[t]
\includegraphics[width=0.5\linewidth]{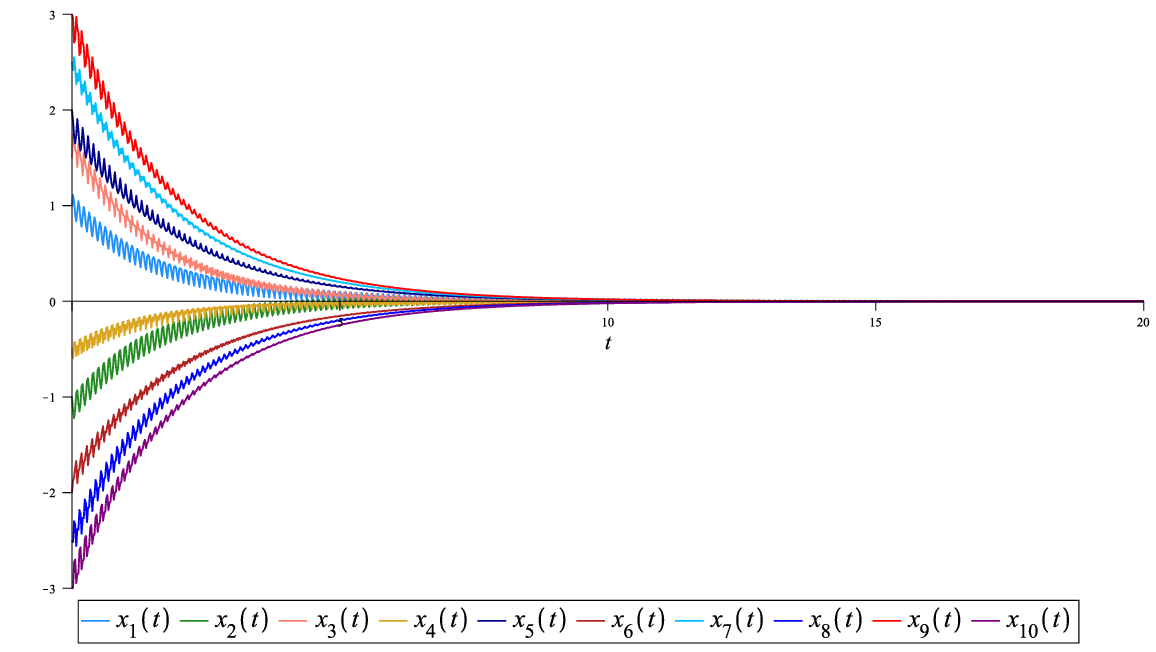} \includegraphics[width=0.5\linewidth]{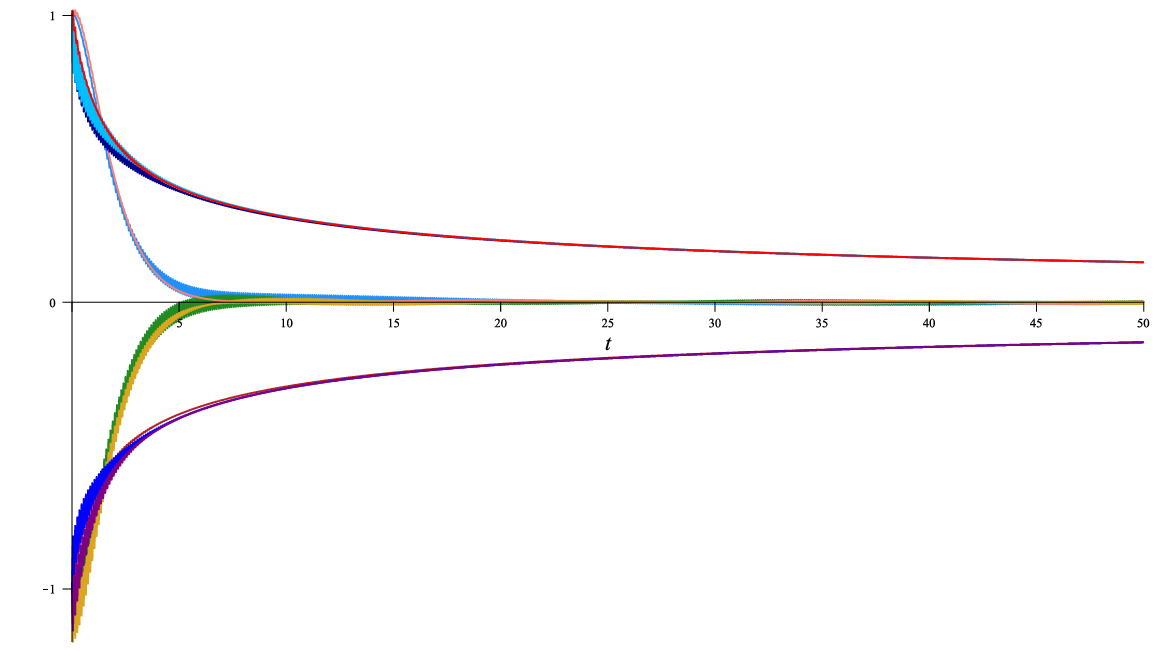}
\caption{{Solution components of the closed-loop system~\eqref{brockett},~\eqref{brockett_feedback},~\eqref{v_functions}   under condition~\eqref{brockett_stabcond_12} with $p=1$ (left) and~\eqref{brockett_stabcond_m} with $p=3/2$ (right).
}}
\end{figure*}
\begin{figure*}[t]
\includegraphics[width=0.5\linewidth]{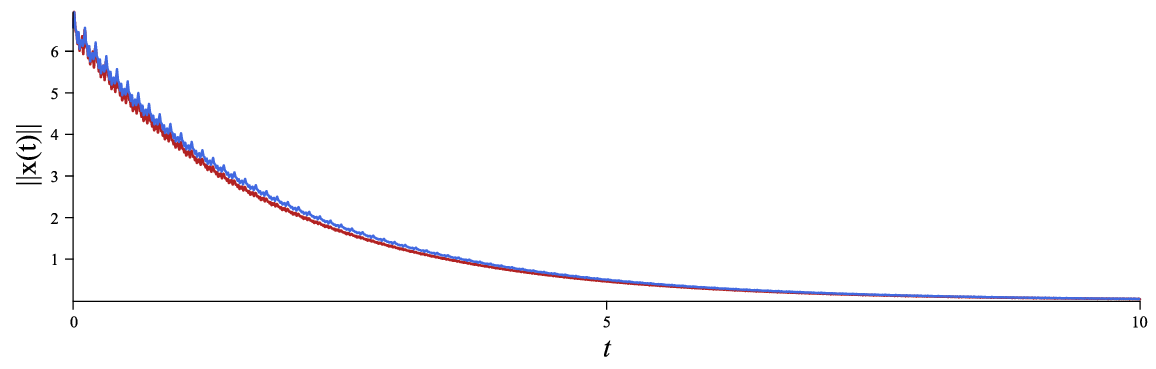} \includegraphics[width=0.5\linewidth]{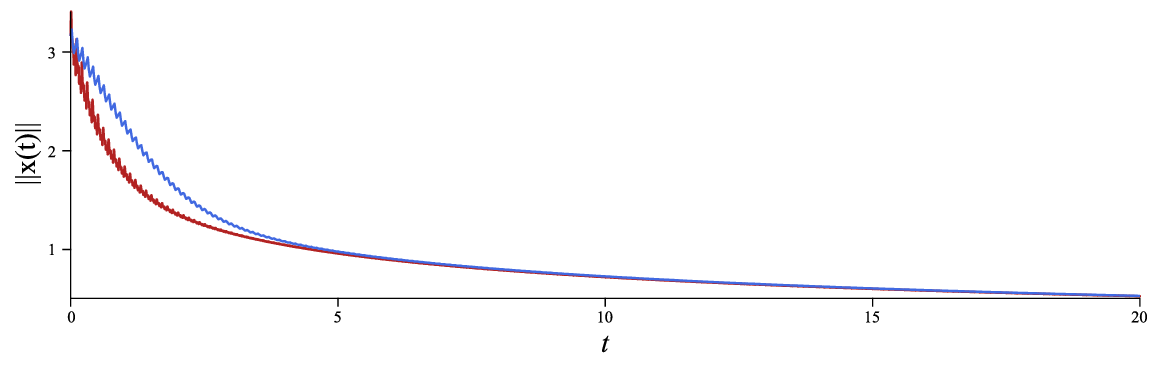}
\caption{{Norms of the classical (blue) and $\pi_\varepsilon$-solution (red) of the closed-loop system~\eqref{brockett},~\eqref{brockett_feedback},~\eqref{v_functions}.}}
\end{figure*}
The solution $x(t)$ of system~\eqref{brockett_closedloop} with the initial data $x(0)=x^0\in{\mathbb R}^{10}$ is represented by the Chen--Fliess series~\eqref{Chen-Fliess}:
$$
x(\varepsilon) = x^0 + \varepsilon\Big( g_0(x^0) + {\gamma}^2\sum_{I=(ij)\in S}[g_i^I,g_j^I](x^0) \Big) + O(\varepsilon^{3/2}).
$$
It is easy to see that $V(x)$ is a weak CLF, but not a strict CLF for system~\eqref{brockett}.
According to the general scheme presented in Section~\ref{gensec}, we design the functions $v_j^I (x)$ such that
$
W(x) = L_{g_0} V(x) +{\gamma}^2 L_{[g_1,g_2]} V(x)
$
is negative definite. Note that
$
\alpha(x)=L_{g_0} V(x) = - (x_1^2 + x_2^2+x_3^2+x_4^2)\le 0,
$
so that condition a) of Lemma~2 is satisfied.
Let us define
{\small \begin{equation}\label{v_functions}
\begin{aligned}
 &v_i^{I} (x)  =\sqrt{|v^{I}(x)|},\, v_j^{I} (x)  =\sqrt{|v^{I}(x)|} {\textrm sign}(v^{I}(x)),\\
& v^{(12)}(x)=-\tfrac{ {\textrm sign} (x_5)}2|x_5|^{2{p}-1},\,v^{(13)}(x)=-\tfrac{ {\textrm sign} (x_6)}2|x_6|^{2{p}-1},\\
&v^{(14)}(x)=-\tfrac{ {\textrm sign} (x_7)}2|x_7|^{2{p}-1},\,v^{(23)}(x)=-\tfrac{ {\textrm sign} (x_8)}2|x_8|^{2{p}-1},\\
&v^{(24)}(x)=-\tfrac{ {\textrm sign} (x_9)}2|x_9|^{2{p}-1},\,v^{(34)}(x)=-\tfrac{ {\textrm sign} (x_{10})}2|x_{10}|^{2{p}-1}.
\end{aligned}
\end{equation}
}
It is easy to see that the above functions $v_i^I(x)$ and $v_j^I(x)$ are sufficiently many times continuously differentiable
under a suitable choice of ${p}$.
Moreover, according to the notation in Lemma~2,
$
\beta(x)=-\|\tilde x\|^{2{p}}+\tfrac{2{p}-1}{4}\Phi(x),
$
where $\tilde x=(x_5,x_6,x_7,x_8,x_9,x_{10})^\top$, and $$\|\Phi(x)\|\le (x_1^2+x_2^2+x_3^2+x_4^2)\|\tilde x\|^{2{p}-2}(1+\|\tilde x\|)\text{ in }\mathbb R^{10}.$$
Thus, for all $x\in\mathbb R^{10}$,
{\begin{equation}\label{W_form}
\begin{aligned}
 W&(x) =- (x_1^2 + x_2^2+x_3^2+x_4^2)-\gamma^2\|\tilde x\|^{2{p}}\\
 &+\tfrac{(2{p}-1)\gamma^2}{4} \Phi(x) \le- (x_1^2 + x_2^2+x_3^2+x_4^2)\\
  \times&\Big(1-\tfrac{(2{p}-1)\gamma^2}{4}\|\tilde x\|^{2{p}-2}(1+\|\tilde x\|)\Big) -\gamma^2\|\tilde x\|^{2{p}}.
\end{aligned}
\end{equation}
The formal derivation of $W(x)$ is valid for $x_j\neq 0$, $j\in\{5,6,\dots,10\}$, and ${p}\in [1,\frac32]$, and for any $x_j$ if ${p} > \frac32$. Note that the resulting expression~\eqref{W_form} is well-defined for all $x\in{\mathbb R}^{10}$ and ${p}\ge \frac32$.
It can be shown that condition b) of Lemma~2 is satisfied in any bounded domain $D$ if ${p}<\frac54$, and condition~ii) holds for $D=\mathbb R^{10}$ if $p\ge \frac54$. Thus, the following corollary of Theorem~1 summarizes this case study: {\em there exist $\varepsilon>0$, $\gamma>0$, and $\delta>0$ such that each solution $x(t)$ of the closed-loop  system~\eqref{brockett},~\eqref{brockett_feedback},~\eqref{v_functions} with $x(0)\in B_\delta(0)$ tends to zero as $t\to+\infty$.}

In the considered example, it is easy to see how to define the gain parameters to ensure that $W(x)$ is negative definite. In particular,  $W(x) $ is negative definite in the domain $D_H=\{x\in\mathbb R\,\vert\; \|\tilde x\| < H\}$, $H>0$, if the following conditions are satisfied:
\begin{equation}\label{brockett_stabcond_m}
{p} > 1,\;   \gamma  < \Big(0, \tfrac{2}{\sqrt{(2{p}-1) H^{2{p}-1}(1+H)}} \Big),
\end{equation}
Moreover, $W(x)$ is negative definite in ${\mathbb R}^{10}$,  provided that
\begin{equation}\label{brockett_stabcond_12}
{p}=1,\; \gamma \in \big(0, \sqrt{2 } \big).
\end{equation}


To illustrate that the proposed feedback control ensures asymptotic stability of the multidimensional Brockett integrator,
we perform numerical integration of the closed-loop system~\eqref{brockett},~\eqref{brockett_feedback},~\eqref{v_functions} under conditions~\eqref{brockett_stabcond_m} and~\eqref{brockett_stabcond_12}.
Fig.~1 illustrates the behavior of the closed-loop system with  {$p=1$}
(left) and {$p=3/2$} (right).
As one can see, in both cases, the solutions tend to zero for large $t$, while the norm $\|x(t)\|$  exhibits  exponential convergence to $0$ if ${p}=1$, and a slower (polynomial) convergence if ${p}=\frac32$. In both cases, we set $\varepsilon=0.1$ and $\gamma=0.5$. The initial conditions are chosen as
$
x(0)=\left(1,-1,1.5,-0.5,2,-2,2.5,-2.5,3,-3\right)^\top,
$
in the first case, and
$
x(0)=\left(1,-1,\dots,1,-1\right)^\top
$
in the second one.
{
To compare the behavior of the classical and $\pi_\varepsilon$-solutions, we also present Fig.~2, which shows the plots of the norms of the corresponding solutions $\|x(t)\|$.
As seen in Fig.~2, the difference between the $\pi_\varepsilon$-solution (red graph) and the classical one (blue graph) is not significant, so that both stabilization schemes perform efficiently in this example.
}

\section{CONCLUSIONS AND FUTURE WORK}\label{secconc}
Our main theoretical contribution is formulated in Theorem~1 and illustrated by the multidimensional Brockett integrator.
One of the key requirements in Theorem~1 and Theorem~2 is the $C^2$ assumption on the control functions and the relevant vector fields $g_0$ and $g_k^I$.
As it can be seen in Section~\ref{secex}, this regularity property holds if the exponent~${p}$,
appearing in the Lyapunov function candidate, is suitably
chosen.

We expect to relax this regularity assumption while still ensuring \textit{practical} stabilizability, meaning attraction to a neighborhood of the origin,
where the neighborhood’s radius decreases as $\varepsilon$
 decreases.  This result can be established similarly to the proof of practical stability in~\cite[Theorem~3, Assertion~I]{GZE18}.
 Additionally, we intend to explore the problem of practical stabilization when
condition iii) of Theorem~\ref{prop_general} is violated.
In future work,
we {also aim to develop a more systematic approach for constructing $v^I_k(x)$ in~\eqref{feedback_general}, potentially reducing reliance on symbolic differentiation and symbolic inversion of the matrix $F(x)$ enabling application to more general settings (with the use of Lemma~2.4). }

{
Although the proposed stabilization scheme is developed theoretically and illustrated using a nilpotent example, it holds strong potential for practical applications, particularly in the control of mobile robots. These systems often exhibit a degree of nonholonomy 2, and the scheme can be applied directly to their original kinematic models or to their nilpotent approximations, such as those constructed using the method described in~\cite{bellaiche2005geometry}. This flexibility makes our approach promising for real-world robotic systems.}


\end{document}